\documentclass{amsart} %[draft]

\usepackage{amsmath, amssymb}
\usepackage{amsfonts,hyperref,caption}
\usepackage{cite}
\theoremstyle{plain}

\newtheorem{thm}{Theorem}[section]
\newtheorem{proposition}[thm]{Proposition}
\newtheorem{lemma}[thm]{Lemma}
\newtheorem{corollary}[thm]{Corollary}
\theoremstyle{definition}
\newtheorem{definition}[thm]{Definition}
\newtheorem{example}[thm]{Example}
\newtheorem{remark}[thm]{Remark}
\def\Hom{\operatorname{Hom}}   \def\ri{\operatorname{ri}}
\def\Supp{\operatorname{Supp}} \def\Ann{\operatorname{Ann}}
\def\HF{\operatorname{HF}}     \def\HP{\operatorname{HP}}
\newcommand{\bbQ}{\ensuremath{\mathbb Q}}
\newcommand{\bbZ}{\ensuremath{\mathbb Z}}
\newcommand{\bbN}{\ensuremath{\mathbb N}}
\newcommand{\bbP}{\ensuremath{\mathbb P}}
\newcommand{\bbX}{\ensuremath{\mathbb X}}
\newcommand{\bbY}{\ensuremath{\mathbb Y}}
\newcommand{\bbW}{\ensuremath{\mathbb W}}

\newcommand{\fq}{\mathfrak{q}}
\newcommand{\fm}{\mathfrak{m}}
\newcommand{\fC}{\mathfrak{C}}

\newcommand{\calO}{{\mathcal{O}}}
\newcommand{\calZ}{{\mathcal{Z}}}

\begin{document}

\title[An application of Liaison theory to zero-dimensional schemes]
{An application of Liaison theory to zero-dimensional schemes}

\author{Martin Kreuzer}
\address[Martin Kreuzer]{Fakult\"{a}t f\"{u}r Informatik und Mathematik \\
Universit\"{a}t Passau, D-94030 Passau, Germany}
\email{martin.kreuzer@uni-passau.de}

\author{Tran~N.~K.~Linh}
\address[Tran~N.~K.~Linh]{Department of Mathematics,
Hue University of Education,
34 Le Loi, Hue, Vietnam}
\email{tnkhanhlinh141@gmail.com}

\author{Le~Ngoc~Long}
\address[Le~Ngoc~Long]{Fakult\"{a}t f\"{u}r Informatik und Mathematik \\
Universit\"{a}t Passau, D-94030 Passau, Germany \newline
\hspace*{.5cm} \textrm{and} Department of Mathematics,
Hue University of Education,
34 Le Loi, Hue, Vietnam}
\email{nglong16633@gmail.com}

\author{Nguyen~Chanh~Tu}
\address[Nguyen Chanh Tu]
{Center of Advanced Programs, Danang University of Technology \\
 54 Nguyen Luong Bang, Danang, Vietnam}
\email{nctudut@gmail.com}



\date{\today}

\dedicatory{}

\commby{Tran N.K. Linh, Le Ngoc Long, and Nguyen Chanh Tu}

% -----------------------------------------------------------
\begin{abstract}
Given a 0-dimensional scheme $\bbX$ in a $n$-dimensional
projective space $\bbP^n_K$ over an arbitrary field $K$,
we use Liaison theory to characterize
the Cayley-Bacharach property of~$\bbX$.
Our result extends the result for sets of $K$-rational points
given in \cite{GKR}.
In addition, we examine and bound the Hilbert function
and regularity index of the Dedekind different of~$\bbX$
when $\bbX$ has the Cayley-Bacharach property.
\end{abstract}
% -----------------------------------------------------------
\maketitle
%-----------------------------------------------------------

\section{Introduction}

The theory of Liaison has been used very extensively
in the literature as a tool to study projective varieties
in the $n$-dimensional projective space $\bbP^n_K$.
The initial idea was to start with a projective variety,
and look at its residual variety in a complete intersection.
Since complete intersections are well understood in some sense,
one can get information about the variety from
its residual variety or vice versa, and so
it would be easier to pass to a ``simpler'' variety
instead of considering a complicated one.
This idea has been also generalized by allowing
links by arithmetically Gorenstein schemes
(see, e.g.,\cite{Mig}).
Currently, Liaison theory is an area of active research
\cite{BKM,Cho,DGO,FGM,GKR,KMMNP,MN,MR,PR},
and has many useful applications, for instance,
constructing interesting projective varieties \cite{BKM,MR,PR},
or computing invariants and establishing properties of
projective varieties \cite{Cho,DGO,FGM,GMN}.

In this paper we are interested in applying
the theory of Liaison to investigate
the geometrical structure of 0-dimensional subschemes
of the $n$-dimensional projective space $\bbP^n_K$ over
an arbitrary field~$K$. This approach was introduced by
Geramita \textit{et al.} \cite{GKR} 
in their study of finite sets of $K$-rational points with
the Cayley-Bacharach property. Classically,
a finite set of $K$-rational points $\bbX$ in~$\bbP^n_K$
is called a {\it Cayley-Bachrach scheme} if any hypersurface
of degree less than the regularity index of the
coordinate ring of~$\bbX$ which contains all points
of~$\bbX$ but one automatically contains the last point.
One of main results of~\cite{GKR} is stated as follows:

\begin{thm} \label{thmSec1.1}
Let $\bbW$ be a set of points in $\bbP^n_K$
which is a complete intersection, let $\bbX\subseteq \bbW$,
let $\bbY=\bbW\setminus\bbX$, and let
$I_\bbW$, $I_\bbX$ and $I_\bbY$ denote the homogeneous
vanishing ideals of $\bbW,\bbX$ and $\bbY$
in $P=K[X_0,...,X_n]$, respectively.
Set $\alpha_{\bbY/\bbW} =
\min\{i\in\bbN \mid (I_\bbY/I_\bbW)_i\ne\langle0\rangle\}$.
Then the following conditions are equivalent.
\begin{enumerate}
  \item[(a)] $\bbX$ is a Cayley-Bachrach scheme.
  \item[(b)] A generic element of $(I_\bbY)_{\alpha_{\bbY/\bbW}}$
  does not vanish at any point of~$\bbX$.
  \item[(c)] We have $I_\bbW:(I_\bbY)_{\alpha_{\bbY/\bbW}} = I_\bbX$.
\end{enumerate}
\end{thm}

This result nicely leads to an efficient algorithm for
checking whether a given set $\bbX$ is
a Cayley-Bacharach scheme.
Later investigations of the Cayley-Bacharach property
have included the work of Fouli, Polini, and Ulrich \cite{FPU},
Robbiano \cite{KR1995}, Gold, Little, and Schenck \cite{GLS}, 
and Guardo~\cite{Gua2000}.
Moreover, this property has also been extended
for 0-dimensional schemes in~$\bbP^n_K$ (see \cite{KLL,Kre,KLR,Lon}).
When $\bbX\subseteq \bbP^n_K$ is a 0-dimensional scheme
over an algebraically closed field $K$, Robbiano and 
the first author \cite{KR1995} considered subschemes
of~$\bbX$ of degree $\deg(\bbX)-1$ to show that
the conditions (a) and~(c) of Theorem~\ref{thmSec1.1} are
still equivalent. However, we get no further information
for a generalization of condition (b) in this case.
It is worth noting here that if $K$ is not algebraically closed
then the scheme $\bbX$ may have no subschemes of degree $\deg(\bbX)-1$.
For example, the 0-dimensional scheme
$\bbX = \calZ(2X_0^4+X_0^2X_1^2-X_1^4)\subseteq \bbP^1_{\bbQ}$
is of degree~4, but it has no subscheme of degree~3.

Our focus in this paper is to look at an extension of the
Cayley-Bacharach property and to generalize the above theorem
for 0-dimensional schemes~$\bbX$ in~$\bbP^n_K$
over an arbitrary field~$K$.
In particular, we will look closely at the natural question
whether conditions (a) and (b) of the above theorem are equivalent
for our more general setting. Our approach is to use the notion
of maximal $p_j$-subschemes of~$\bbX$ which are introduced and
studied in the papers~\cite{KL,KLL}.
Also, we discuss a characterization
of the Cayley-Bacharach property of degree $d$ with $d\in\bbN$
in terms of the canonical module of the coordinate ring of~$\bbX$
and apply this result to bound the Hilbert function of the
Dedekind different of~$\bbX$ and determine its regularity index
in some special cases.

This paper is structured as follows. In Section~\ref{Section2},
we introduce the relevant information about Hilbert function,
maximal $p_j$-subschemes, standard set of sep\-arators,
and Liaison techniques.  Especially, we give
an explicit description of the residual scheme in a
0-dimensional arithmetically Gorenstein scheme of a
maximal $p_j$-subscheme of~$\bbX$.
In Section~\ref{Section3}, we prove
the generalization of the results mentioned above
(see Theorems~\ref{thmSec3.4} and \ref{thmSec3.8}).
We also give Example~\ref{examS3.7} to show that
the condition~(b) in Theorem~\ref{thmSec1.1}
is, in general, only a sufficient condition,
not a necessary condition, for $\bbX$
being a Cayley-Bacharach scheme.
In the final section, we characterize
the Cayley-Bacharach property of degree $d$ using
the canonical module of the coordinate ring of~$\bbX$,
and then look at the Hilbert function of the
Dedekind different of~$\bbX$ and its regularity index
when $\bbX$ has the Cayley-Bacharach property of degree $d$.

All examples in this paper were calculated by using
the computer algebraic system ApCoCoA (see~\cite{ApC}).

\medskip\bigbreak
\section{Basic Facts and Notation}
\label{Section2}

Throughout the paper, we work over an arbitrary field $K$.
The $n$-dimensional projective space over $K$ is denoted
by~$\bbP^n_K$ and its homogeneous coordinate ring is the polynomial ring
$P=K[X_0,...,X_n]$ equipped with the standard grading.
Our object of interest is a 0-dimensional subscheme $\bbX$
of~$\bbP^n_K$. Its homogeneous vanishing ideal in~$P$ is denoted
by~$I_\bbX$ and its homogeneous coordinate ring is
given by~$R_\bbX=P/I_\bbX.$
The set of closed points of~$\bbX$ is called the {\it support}
of~$\bbX$ and is denoted by $\Supp(\bbX) = \{p_1,\dots,p_s\}$.
We always assume that $\Supp(\bbX)\cap\calZ(X_0) = \emptyset$.
Under this assumption, the image $x_0$ of~$X_0$ in~$R_\bbX$
is a non-zerodivisor, and hence $R_\bbX$ is a 1-dimensional
Cohen-Macaulay ring.
To each point $p_j \in \Supp(\bbX)$ we have the associated
local ring~$\calO_{\bbX,p_j}$. Its maximal ideal is denoted
by $\fm_{\bbX,p_j}$, and the residue field of~$\bbX$ at~$p_j$
is denoted by~$\kappa(p_j)$.
The {\it degree} of~$\bbX$ is defined as
$\deg(\bbX) = \sum_{j=1}^s \dim_K(\calO_{\bbX,p_j})$.

Given any finitely generated graded $R_\bbX$-module $M$,
the {\it Hilbert function} of $M$ is a map
$\HF_M: \bbZ \rightarrow \bbN$ given by $\HF_M(i)=\dim_K(M_i)$.
The unique polynomial $\HP_M(z) \in \mathbb{Q}[z]$ for which
$\HF_M(i)=\HP_M(i)$ for all $i\gg 0$ is called the
{\it Hilbert polynomial} of~$M$. The number
$$
\ri(M) = \min\{i\in\bbZ\mid \HF_{M}(j)=\HP_M(j)\
\textrm{for all} \ j\geq i\}
$$
is called the {\it regularity index} of $M$ (or of $\HF_M$).
Whenever $\HF_M(i)=\HP_M(i)$ for all $i\in\bbZ$,
we let $\ri(M)=-\infty.$
Instead of $\HF_{R_\bbX}$ we also write $\HF_{\bbX}$
and call it the Hilbert function of $\bbX$.
Its regularity index is denoted by~$r_{\bbX}.$
Note that $\HF_{\bbX}(i)=0$ for $i<0$ and
$$
1=\HF_{\bbX}(0)<\HF_{\bbX}(1)<\cdots<\HF_{\bbX}(r_{\bbX}-1)< \deg(\bbX)
$$
and $\HF_{\bbX}(i)=\deg(\bbX)$ for $i\geq r_{\bbX}$.

\begin{definition}
Let $1\le j\le s$. A subscheme $\bbX'\subsetneq \bbX$ is called
a {\it $p_j$-subscheme} if the following conditions are
satisfied:
\begin{enumerate}
  \item[(a)]
  $\calO_{\bbX',p_k} = \calO_{\bbX,p_k}$ for $k \ne j$.
  \item[(b)] The map
  $\calO_{\bbX,p_j}\twoheadrightarrow\calO_{\bbX',p_j}$
  is an epimorphism.
\end{enumerate}

A $p_j$-subscheme $\bbX'\subseteq\bbX$ is called {\it maximal}
if $\deg(\bbX') = \deg(\bbX)-\dim_K \kappa(p_j)$.
\end{definition}

In case $\bbX$ has $K$-rational support (i.e., all points
$p_1,...,p_s$ are $K$-rational), a maximal $p_j$-subscheme
of~$\bbX$ is nothing but a subscheme $\bbX'\subseteq\bbX$
of degree $\deg(\bbX')=\deg(\bbX)-1$ with
$\calO_{\bbX',p_j} \ne \calO_{\bbX,p_j}$.
According to \cite[Proposition~3.2]{KL}, there is
a 1-1 correspondence between a maximal $p_j$-subscheme $\bbX'$
and an ideal $\langle s_j\rangle$ in~$\calO_{\bbX,p_j}$,
where $s_j$ is an element in the socle
$\Ann_{\calO_{\bbX,p_j}}(\fm_{\bbX,p_j})$ of~$\calO_{\bbX,p_j}$.
The vanishing ideal of the scheme~$\bbX'$ in~$R_\bbX$ is denoted
by~$I_{\bbX'/\bbX}$ and its initial degree is given by
$\alpha_{\bbX'/\bbX} = \min \{\, i\in\bbN \mid
(I_{\bbX'/\bbX})_i \ne \langle 0\rangle \,\}$.
We find a non-zero element
$f_\bbX' \in (I_{\bbX'/\bbX})_i$, $i \ge \alpha_{\bbX'/\bbX}$,
such that
$\tilde{\imath}(f_\bbX') = (0,\dots,0, s_jT_j^i,0,\dots,0)$,
where the map
$$
\tilde{\imath}: R_\bbX \rightarrow Q^h(R_\bbX) \cong
{\textstyle\prod\limits_{j=1}^s}\calO_{\bbX,p_j}[T_j,T_j^{-1}]
$$
is the injection given by
$\tilde{\imath}(f) = (f_{p_1}T_1^i,\dots, f_{p_s}T_s^i)$,
for $f \in (R_\bbX)_i$ with $i\ge 0$, where $f_{p_j}$
is the germ of~$f$ at~$p_j$. Here the ring $Q^h(R_\bbX)$ is
the {\it homogeneous ring of quotients} of~$R_\bbX$ defined as
the localization of~$R_\bbX$ with respect to the set of
all homogeneous non-zerodivisors of~$R_\bbX$
(cf.~\cite[Section~3]{KL}).

Let $\varkappa_j := \dim_K \kappa(p_j)$, and let
$\{e_{j1}, \dots, e_{j\varkappa_j}\}\subseteq \calO_{\bbX,p_j}$
be elements whose residue classes form a $K$-basis of $\kappa(p_j)$.
For $a \in \calO_{\bbX,p_j}$, we set
$$
\mu(a) := \min\{i\in\bbN \,\mid\,
(0,\dots,0,aT_j^i,0,\dots,0) \in \tilde{\imath}(R_\bbX)\}.
$$
Since the restriction map
$\tilde{\imath}|_{(R_\bbX)_{r_\bbX}}:(R_\bbX)_{r_\bbX}
\rightarrow (\prod_{j=1}^s\calO_{\bbX,p_j}[T_j,T_j^{-1}])_{r_\bbX}$
is an isomor\-phism of K-vector spaces, we have
$\mu(a)\le r_\bbX$ for all $a\in \calO_{\bbX,p_j}$.
Using this notation, we recall from \cite[Section~1]{KLL}
the following notion of separators.

\begin{definition}
Let $\bbX'$ be a maximal $p_j$-subscheme as above, and let
$$
f^*_{jk_j} := \tilde{\imath}^{-1}((0,\dots,0,
e_{jk_j}s_{j}T_j^{\mu(e_{jk_j}s_{j})},0,\dots,0))
$$
and $f_{jk_j}=x_0^{r_{\bbX}-\mu(e_{jk_j}s_{j})}f^*_{jk_j}$
for $k_j = 1, \dots, \varkappa_j$.
\begin{enumerate}
 \item[(a)] The set $\{f^*_{j1},\dots,f^*_{j\varkappa_j}\}$
 is called the
 {\it set of minimal separators of $\bbX'$ in $\bbX$}
  with respect to $s_j$ and $\{e_{j1},\dots,e_{j\varkappa_j}\}$.

 \item[(b)] The set $\{f_{j1},\dots,f_{j\varkappa_j}\}$
 is called the
 {\it standard set of separators of $\bbX'$ in~$\bbX$}
 with respect to $s_j$ and $\{e_{j1},\dots,e_{j\varkappa_j}\}$.

 \item[(c)] The number
   $$
   \mu_{\bbX'/\bbX}:=\max\{\, \deg(f^*_{jk_j})
   \mid k_j = 1, \dots, \varkappa_j \,\}
   $$
   is called the {\it maximal degree of a minimal
   separator of~$\bbX'$ in~$\bbX$}.
\end{enumerate}
\end{definition}

\begin{remark} \label{remSec2.3}
Let $\bbX'$ be a maximal $p_j$-subscheme of~$\bbX$.
\begin{enumerate}
  \item[(a)]
  The maximal degree of a minimal separator of~$\bbX'$ in~$\bbX$
  depends neither on the choice of the socle element~$s_j$
  nor on the specific choice of~$\{e_{j1},\dots,e_{j\varkappa_j}\}$
  (see \cite[Lemma~4.4]{KL}).
  Moreover, we have $\mu_{\bbX'/\bbX}\le r_\bbX$.

  \item[(b)]
  For $k_j =1,\dots,\varkappa_j$, let $F_{jk_j}$
  (respectively, $F^*_{jk_j}$) be a representative
  of $f_{jk_j}$ (respectively, $f^*_{jk_j}$) in $P$.
  We also say that the set $\{F_{j1},\dots,F_{j\varkappa_j}\}$
  is a standard set of separators of~$\bbX'$ in~$\bbX$
  and the set $\{F^*_{j1},\dots,F^*_{j\varkappa_j}\}$
  is a set of minimal separators of~$\bbX'$ in~$\bbX$.

  \item[(c)]
  According to \cite[Proposition~2.5(c)]{KLL},
  one may choose a set of minimal separators
  $\{f^*_{j1},\dots,f^*_{j\varkappa_j}\}$ of~$\bbX'$ in~$\bbX$
  such that
  $$
  (I_{\bbX'/\bbX})_i = \big\langle\,
  x_0^{i-\deg(f^*_{jk_j})}f^*_{jk_j} \,\mid\, \deg(f^*_{jk_j})\le i
  \,\big\rangle_K
  $$
  for all $i\ge 0$.
\end{enumerate}
\end{remark}

Recall that a 0-dimensional scheme $\bbX$ is called a
{\it complete intersection} if $I_\bbX$ can be generated
by $n$ homogeneous polynomials in~$P$, and it is called
an {\it arithmetically Gorenstein scheme} if $R_\bbX$
is a Gorenstein ring. Note that every complete intersections
are arithmetically Gorenstein, however, except for the
case $n=2$, an arithmetically Gorenstein scheme is
not a complete intersection in general
(see~\cite[Example~2.12]{KL}).

In what follows, we let $\bbW \subseteq \bbP^n_K$ be
a 0-dimensional arithmetically Gorenstein scheme,
let $\bbX$ be a subscheme of~$\bbW$, and let
$I_{\bbX/\bbW}$ be the ideal of $\bbX$ in $R_\bbW$.
Then the homogeneous ideal
$\Ann_{R_\bbW}(I_{\bbX/\bbW})\subseteq R_\bbW$
is saturated and defines a 0-dimensional subscheme~$\bbY$ of~$\bbW$.

\begin{definition}\label{defnSec2.4}
\begin{enumerate} 
  \item[(a)]
  The subscheme $\bbY \subseteq \bbW$
  which is defined by the saturated homogeneous ideal
  $I_{\bbY/\bbW} = \Ann_{R_\bbW}(I_{\bbX/\bbW})$
  is said to be the {\it residual scheme}
  of~$\bbX$ in~$\bbW$. We also say that
  $\bbX$ and $\bbY$ are {\it (algebraically) linked}
  by~$\bbW$.

  \item[(b)]
  Two linked schemes $\bbX$ and $\bbY$ by~$\bbW$
  are said to be {\it geometrically linked} by~$\bbW$
  if they have no common irreducible component.
\end{enumerate}
\end{definition}

\begin{remark}
From the point of view of the saturated ideals,
the schemes $\bbX$ and $\bbY$ are geometrically linked
by~$\bbW$ if and only if $I_{\bbW} = I_{\bbX} \cap I_{\bbY}$
and neither $I_{\bbX}$ nor $I_{\bbY}$ is contained in
any associated prime of the other (see~\cite[Section 5.2]{Mig}).
In this case, if we write $\Supp(\bbX) = \{p_1, \dots, p_s\}$ and
$\Supp(\bbY) = \{p'_{1}, \dots, p'_{t}\}$, then we have
$\Supp(\bbW) = \{p_1, \dots, p_s, p'_{1}, \dots, p'_{t}\}$
and $\Supp(\bbX)\cap\Supp(\bbY) = \emptyset$.
In particular, we have $\calO_{\bbW,p_j}=\calO_{\bbX,p_j}$
for $j=1,\dots,s$ and $\calO_{\bbW,p'_j}=\calO_{\bbY,p'_j}$
for $j = 1, \dots, t$.
\end{remark}

First we collect some useful results about the linked schemes
$\bbX$ and $\bbY$ by the arithmetically Gorenstein scheme $\bbW$.

\begin{proposition}\label{propSec2.5}
\begin{enumerate}
  \item[(a)]
  \!We have $I_{\bbX/\bbW} = \Ann_{R_\bbW}(I_{\bbY/\bbW})$.

  \item[(b)]
  We have $\deg(\bbW) = \deg(\bbX) + \deg(\bbY)$
  and $r_{\bbW} = r_{\bbX}+\alpha_{\bbY/\bbW}
  = r_{\bbY}+\alpha_{\bbX/\bbW}$.

  \item[(c)] The Hilbert function of $I_{\bbY/\bbW}$ satisfies
  $$
  \HF_{I_{\bbY/\bbW}}(i) = \deg(\bbX)-\HF_{\bbX}(r_\bbW-i-1)
  \quad \mbox{for all $i\in\bbZ$.}
  $$
\end{enumerate}
\end{proposition}
\begin{proof}
Claims (a) and (b) follow from~\cite{DGO}.
To prove (c), we use (a) and (b) and~\cite[Proposition~2.2.9]{GW}
to get the following sequence of isomorphism of graded $R_\bbW$-modules
$$
\begin{aligned}
I_{\bbY/\bbW}(\alpha_{\bbY/\bbW})
&= \Ann_{R_\bbW}(I_{\bbX/\bbW})(\alpha_{\bbY/\bbW})
\cong \Hom_{R_\bbW}(R_\bbW/I_{\bbX/\bbW},R_\bbW)(r_\bbW-r_\bbX)\\
&\cong \Hom_{R_\bbW}(R_\bbW/I_{\bbX/\bbW},
 R_\bbW(r_\bbW))(-r_\bbX) \\
&\cong \Hom_{R_\bbW}(R_\bbX,
\Hom_{K[x_0]}(R_\bbW,K[x_0]))(-r_\bbX)\\
&\cong \Hom_{K[x_0]}(R_\bbX,K[x_0])(-r_{\bbX}).
\end{aligned}
$$
It is well known (cf.~\cite[Lemma~1.3]{Kre}) that
$\HF_{\Hom_{K[x_0]}(R_\bbX,K[x_0])}(i)=\deg(\bbX)-\HF_\bbX(-i-1)$
for all $i\in \bbZ$. Hence we get the desired formula
for the Hilbert function of~$I_{\bbY/\bbW}$ and claim (c)
follows.
\end{proof}

In the following we shall use ``$^{\overline{\;\;\,}}$''
to denote residue classes modulo~$X_0$.

\begin{lemma} \label{lemSec2.6}
For every $d \in \{1, \dots, r_{\bbX}\}$, we have
$(\overline{I}_\bbW)_{r_{\bbW}}:
(\overline{I}_\bbY)_{\alpha_{\bbY/\bbW}+(r_{\bbX}-d)}
= (\overline{I}_\bbX)_d$.
\end{lemma}

\begin{proof}
Clearly, we have $I_{\bbX}\cdot I_{\bbY} \subseteq I_{\bbW}$.
This implies
$(\overline{I}_\bbX)_d \subseteq (\overline{I}_\bbW)_{r_\bbW}
\!:\! (\overline{I}_\bbY)_{\alpha_{\bbY/\bbW}+(r_\bbX-d)}$.
For the other inclusion, let $f \in (\overline{I}_\bbW)_{r_\bbW}:
(\overline{I}_\bbY)_{\alpha_{\bbY/\bbW}+(r_\bbX-d)}$.
In~$\overline{R}_\bbW = R_\bbW / \langle x_0 \rangle$,
we have $\overline{f} \in (\Ann_{\overline{R}_\bbW}
((\overline{I}_{\bbY/\bbW})_{\alpha_{\bbY/\bbW}+(r_{\bbX}-d)}))_d$.
Since $\bbW$ is arithmetically Gorenstein, the ring $\overline{R}_\bbW$
is a $0$-dimensional local Gorenstein ring with socle
$(\overline{R}_\bbW)_{r_\bbW} \cong K$.
Thus we can argue in the same way as Lemma~4.1 and Proposition~4.3.a
of~\cite{GKR} to get
$$
\begin{aligned}
(\Ann_{\overline{R}_\bbW}((\overline{I}_{\bbY/\bbW}
)_{\alpha_{\bbY/\bbW}+(r_{\bbX}-d)}))_d
&= (\Ann_{\overline{R}_\bbW}((\overline{I}_{\bbY/\bbW})_{r_\bbW-d}))_d\\
&= (\Ann_{\overline{R}_\bbW}(\overline{I}_{\bbY/\bbW}))_d
= (\overline{I}_{\bbX/\bbW})_d.
\end{aligned}
$$
Consequently, we have
$\overline{f} \in (\overline{I}_{\bbX/\bbW})_d$,
and hence $f \in (\overline{I}_{\bbX})_d$, as desired.
\end{proof}

The next lemma follows for instance from
\cite[3.15 and 16.38-40]{Lam1998}.

\begin{lemma} \label{lemSec2.9}
Let $A/K$ be a finite Gorenstein algebra.
\begin{enumerate}
  \item[(a)] There is a non-degenerate $K$-bilinear form
  $\Phi : A\times A \rightarrow K$
   with the property that $\Phi(xy,z)=\Phi(x,yz)$
   for all $x,y,z \in A$.

  \item[(b)] Let $I$ be a non-zero ideal of~$A$, and let
   $I^0 = \{\, x \in A \, \mid \, \Phi(I,x)=0 \,\}$.
   Then we have $\Ann_A(I) = I^0$ and
   $\dim_KI + \dim_K \Ann_A(I) = \dim_K A$.
\end{enumerate}
\end{lemma}

A concrete description of the residual scheme in~$\bbW$
of a maximal $p_j$-subscheme of~$\bbX$ is given by the
following proposition.

\begin{proposition} \label{propSec2.10}
Let $\bbW\subseteq\bbP^n_K$ be a $0$-dimensional arithmetically
Gorenstein scheme, let $\bbX$ and $\bbX'$ be subschemes
of~$\bbW$, let $\bbY$ and $\bbY'$ be the residual schemes
of~$\bbX$ and $\bbX'$ in~$\bbW$ respectively,
and let $p_j \in \Supp(\bbX)$.
Then $\bbX'$ is a (maximal) $p_j$-subscheme of~$\bbX$ if and only if
$\bbY'$ contains $\bbY$ as a (maximal) $p_j$-subscheme.
\end{proposition}

\begin{proof}
As sets, we have $\Supp(\bbW)=\Supp(\bbX) \cup \Supp(\bbY)$
by~\cite[Proposition~5.2.2]{Mig}. Let us write
$\Supp(\bbX) = \{p_1, \dots, p_t\}$,
$\Supp(\bbY) = \{p_{s+1}, \dots, p_u\}$, and
$\Supp(\bbW) =
\{p_1, \dots, p_s, p_{s+1}, \dots, p_t, p_{t+1}, \dots, p_u\}$
with $s\le t\le u$.  Then there exist ideals
$\fq_{s+1} \subseteq \calO_{\bbW, p_{s+1}}$, $\dots$,
$\fq_{t} \subseteq \calO_{\bbW, p_{t}}$ such that
$$
\calO_{\bbX,p_j} =
\begin{cases}
\calO_{\bbW, p_j} & \mbox{ for } j = 1, \dots, s, \\
\calO_{\bbW, p_j}/\fq_j & \mbox{ for } j=s+1,\dots,t,\\
\langle0\rangle & \mbox{ for } j = t+1, \dots, u. \\
\end{cases}
$$
Consider the map
$\theta : R_\bbW \rightarrow \prod_{j=1}^u \calO_{\bbW, p_j}$
given by $f \mapsto (f_{p_1},\dots,f_{p_u})$.
According to~\cite[Lemma~1.1]{Kre}, the restriction
$\theta |_{(R_\bbW)_i}$ is an injection for $0\le i<r_{\bbW}$
and is an isomorphism for all $i \ge r_{\bbW}$.
By Proposition~\ref{propSec2.5}(b), we have
$r_\bbX \le r_\bbW$ and
$\HF_{I_{\bbX/\bbW}}(i)=\deg(\bbW)-\deg(\bbX) =\deg(\bbY)$
for $i\ge r_\bbW$. Consequently, we get
$$
\theta(I_{\bbX/\bbW}) = \theta((I_{\bbX/\bbW})_{r_{\bbW}}) =
\langle0\rangle \times\cdots\times\langle0\rangle
\times \fq_{s+1}\times\cdots\times\fq_t
\times \calO_{\bbW, p_{t+1}}\times\cdots\times \calO_{\bbW, p_{u}}
$$
and $\dim_K \theta(I_{\bbX/\bbW}) = \deg(\bbY)$.
In $\prod_{j=1}^u \calO_{\bbW, p_j}$, we set
$$
\Lambda := \calO_{\bbW, p_{1}}\times\cdots\times \calO_{\bbW, p_{s}}
\times \Ann_{\calO_{\bbW, p_{s+1}}}(\fq_{s+1})\times\cdots\times
\Ann_{\calO_{\bbW, p_{t}}}(\fq_t)
\times \langle0\rangle\times\cdots\times\langle0\rangle.
$$
It follows from Lemma~\ref{lemSec2.9} that
$\deg(\bbW) = \dim_K \Lambda + \dim_K \theta(I_{\bbX/\bbW})$.
This implies $\dim_K \Lambda = \deg(\bbW) - \deg(\bbY)$.
Now we want to verify that $\theta(I_{\bbY/\bbW}) = \Lambda$.
Since $\dim_K \theta(I_{\bbY/\bbW})=\deg(\bbW)-\deg(\bbY)
=\dim_K \Lambda$, it suffices to show that
$\Lambda\subseteq \theta(I_{\bbY/\bbW})$.
Let $i\ge 0$, let $f \in (R_\bbW)_i\setminus\{0\}$ such that
$\theta(f) \in \Lambda$, and let
$g \in (I_{\bbX/\bbW})_{k}\setminus\{0\}$ with
$k \ge \alpha_{\bbX/\bbW}$.
Then we have $f\cdot g \in (R_\bbW)_{i+k}$ and
$\theta(f\cdot g) = \mathbf{0}$.
This implies $f\cdot g = 0$. Consequently, we have
$f \in \Ann_{R_\bbW}(I_{\bbX/\bbW})=I_{\bbY/\bbW}$,
and hence $\Lambda\subseteq \theta(I_{\bbY/\bbW})$.

Now we assume that $\bbX'$ is a $p_j$-subscheme of~$\bbX$.
Then we have $\calO_{\bbX',p_k} = \calO_{\bbX,p_k}$
for all $k \in \{1, \dots, u\}\setminus\{j\}$ and
$\calO_{\bbX',p_j} = \calO_{\bbW,p_j}/\fq'_j$.
We distinguish the following two cases.

\medskip
\noindent{\bf Case (${\bf a}$)}\quad
Suppose that $1 \le j \le s$.
We see that $\fq'_j \ne \langle 0 \rangle$ and
$$
\begin{aligned}
\theta(I_{\bbX'/\bbW}) &=
\{0\}\times\cdots\times \fq'_j \times\cdots\times \{0\}
\times \fq_{s+1}\times\cdots\times\fq_t
\times\calO_{\bbW, p_{t+1}}\times\cdots\times\calO_{\bbW, p_{u}} \\
\theta(I_{\bbY'/\bbW}) &=
\calO_{\bbW, p_{1}}\times\cdots\times\Ann_{\calO_{\bbW, p_{j}}}(\fq'_{j})
\times\cdots\times  \calO_{\bbW, p_{s}}
\times \Ann_{\calO_{\bbW, p_{s+1}}}(\fq_{s+1})\times\cdots \\
&\quad \times \Ann_{\calO_{\bbW, p_{t}}}(\fq_t)
\times \{0\}\times\cdots\times\{0\}.
\end{aligned}
$$
This implies $\calO_{\bbY', p_k} = \calO_{\bbY,p_k}$
for $k \ne j$ and
$\calO_{\bbY',p_j}=\calO_{\bbW,p_j}/\Ann_{\calO_{\bbW, p_{j}}}(\fq'_{j})
\ne \langle 0 \rangle = \calO_{\bbY,p_j}$.
Hence $\bbY$ is a $p_j$-subscheme of~$\bbY'$.

\medskip
\noindent{\bf Case (${\bf b}$)}\quad
Suppose that $s+1 \le j \le t$.
We have $\fq'_j \supsetneq \fq_j$ and
$$
\begin{aligned}
\theta(I_{\bbX'/\bbW}) &=
\{0\}\times\cdots\times \{0\}
\times \fq_{s+1}\times\cdots\times\fq'_j \times\cdots\times\fq_t
\times\calO_{\bbW, p_{t+1}}\times\cdots\times\calO_{\bbW,p_{u}}\\
\theta(I_{\bbY'/\bbW}) &=
\calO_{\bbW, p_{1}}\times\cdots\times \calO_{\bbW, p_{s}}
\times \Ann_{\calO_{\bbW, p_{s+1}}}(\fq_{s+1})\times \cdots
\times \Ann_{\calO_{\bbW, p_{j}}}(\fq'_{j}) \times\cdots \\
&\quad \times \Ann_{\calO_{\bbW, p_{t}}}(\fq_t)
\times \{0\}\times\cdots\times\{0\}.
\end{aligned}
$$
This implies that $\calO_{\bbY', p_k} = \calO_{\bbY,p_k}$ for $k \ne j$ and
$\calO_{\bbY',p_j} = \calO_{\bbW,p_j}/\Ann_{\calO_{\bbW, p_{j}}}(\fq'_{j})
\ne \calO_{\bbW,p_j}/\Ann_{\calO_{\bbW, p_{j}}}(\fq_{j}) = \calO_{\bbY,p_j}$.
Again the scheme $\bbY$ is a $p_j$-subscheme of~$\bbY'$.

Conversely, if $\bbY$ is a $p_j$-subscheme of~$\bbY'$,
where $p_j \in \Supp(\bbX)$, an analogous argument
as above yields that $\bbX'$ is a $p_j$-subscheme of~$\bbX$.
\end{proof}

\medskip\bigbreak
\section{Cayley-Bacharach Property and Liaison} \label{Section3}

In this section we use liaison techniques to characterize
the Cayley-Bacharach property of a 0-dimensional scheme
$\bbX$ in~$\bbP^n_K$.
First we recall the notions of the degree of a point
in~$\bbX$ and the Cayley-Bacharach property
(see \cite[Section~4]{KLL}).

\begin{definition}
Let $d \ge 0$, let $\bbX \subseteq \bbP^n_K$ be a $0$-dimensional
scheme, and let $\Supp(\bbX)=\{p_1,\dots,p_s\}$.
\begin{enumerate}
  \item[(a)]
  For $1\le j\le s$, the {\it degree of $p_j$ in $\bbX$}
  is defined as
  $$
  \deg_{\bbX}(p_j) := \min\big\{\, \mu_{\bbX'/\bbX} \; \big| \;
  \bbX' \ \textrm{is a maximal $p_j$-subscheme of $\bbX$} \,\big\},
  $$
  where $\mu_{\bbX'/\bbX}$ is the maximal degree of a minimal
  separator of~$\bbX'$ in~$\bbX$.

  \item[(b)]
  We say that $\bbX$ has the
  {\it Cayley-Bacharach property of degree $d$}
  (in short, $\bbX$ has CBP($d$)) if $\deg_{\bbX}(p_j)\ge d+1$
  for every $j\in\{1,\dots,s\}$.
  In the case that $\bbX$ has CBP($r_\bbX-1$) we also say that
  $\bbX$ is a {\it Cayley-Bacharach scheme}.
\end{enumerate}
\end{definition}

According to Remark~\ref{remSec2.3}(a), we have
$0\le \deg_{\bbX}(p_j)\le r_\bbX$. So,
the number $r_{\bbX}-1$ is the largest degree $d\ge0$
such that $\bbX$ can have CBP($d$).
Hence it suffices to consider the Cayley-Bacharach property
in degree $d \in \{0, \dots, r_{\bbX}-1\}$.
Using standard sets of separators of~$\bbX$, we can
characterize the Cayley-Bacharach property as follows
(see \cite[Proposition~4.3]{KLL}).

\begin{proposition}\label{propSec3.2}
Let $0 \le d <r_{\bbX}$,
let $\Supp(\bbX) = \{p_1, \dots, p_s\}$, and let
$\varkappa_j = \dim \kappa(p_j)$.
Then the following statements are equivalent.
\begin{enumerate}
  \item[(a)] The scheme $\bbX$ has CBP($d$).

  \item[(b)]
  If $\bbX' \subseteq \bbX$ is a maximal $p_j$-subscheme
  and $\{f_{j1},\dots,f_{j\varkappa_j}\}\subseteq R_\bbX$ is a standard set
  of separators of~$\bbX'$ in~$\bbX$, then there exists
  $k_j\in\{1 \dots, \varkappa_j\}$ such that
  $x_0^{r_{\bbX}-d} \nmid f_{jk_j}$.

  \item[(c)]
  If $\bbX' \subseteq \bbX$ is a maximal $p_j$-subscheme
  and $\{F_{j1},\dots,F_{j\varkappa_j}\}\subseteq P$
  is a standard set of separators of~$\bbX'$ in~$\bbX$,
  then there exists $k_j\in\{1 \dots, \varkappa_j\}$ such that
  $F_{jk_j}\notin\langle X_0^{r_{\bbX}-d},(I_\bbX)_{r_\bbX}\rangle_P$.

  \item[(d)]
  For all $p_j \in \Supp(\bbX)$, every maximal
  $p_j$-subscheme $\bbX' \subseteq \bbX$ satisfies
  $$
  \dim_K(I_{\bbX'/\bbX})_{d} < \varkappa_j.
  $$
\end{enumerate}
\end{proposition}

Now we give two useful lemmas that will be used in the
proof of results below.

\begin{lemma} \label{lemSec3.3}
Let $\bbW \subseteq \bbP^n_K$ be a $0$-dimensional
arithmetically Gorenstein scheme, let $\bbX$
be a subscheme of~$\bbW$ with its residual
scheme~$\bbY$, and let $0 \le d < r_{\bbX}$.
Furthermore, let $\bbX'\subseteq\bbX$ be a maximal
$p_j$-subscheme, and let
$\{F_{j1},\dots,F_{j\varkappa_j}\}\subseteq P_{r_{\bbX}}$
be a standard set of separators of~$\bbX'$ in~$\bbX$.
Suppose that
$\langle F_{j1}, \dots, F_{j\varkappa_j} \rangle_K \nsubseteq
\langle X_0^{r_{\bbX}-d}, (I_{\bbX})_{r_{\bbX}} \rangle_P$
and $\langle F_{j1}, \dots, F_{j\varkappa_j} \rangle_K \subseteq
\langle X_0^{r_{\bbX}-d-1}, (I_{\bbX})_{r_{\bbX}} \rangle_P$,
and write $F_{jk_j} = F'_{jk_j} + X_0^{r_{\bbX}-d-1}G_{jk_j}$ with
$F'_{jk_j} \in (I_{\bbX})_{r_{\bbX}}$ and $G_{jk_j}\in P_{d+1}$.

Then there is $k_j \in \{1, \dots, \varkappa_j\}$ such that
$G_{jk_j}\notin (I_{\bbW})_{r_{\bbW}}:(I_{\bbY})_{r_{\bbW}-d-1}$.
\end{lemma}

\begin{proof}
Suppose that $G_{jk_j}\in(I_\bbW)_{r_\bbW}:(I_{\bbY})_{r_\bbW-d-1}$
for all $k_j = 1, \dots, \varkappa_j$.
By modulo $X_0$ we have
$\overline{G}_{jk_j}(\overline{I}_\bbY)_{r_\bbW-d-1}
\subseteq (\overline{I}_\bbW)_{r_\bbW}$.
Note that $r_\bbW = \alpha_{\bbY/\bbW}+r_\bbX$
by Proposition~\ref{propSec2.5}(b). Thus Lemma~\ref{lemSec2.6}
yields that $\overline{G}_{jk_j} \in (\overline{I}_\bbX)_{d+1}$.
This allows us to write $G_{jk_j} = G'_{jk_j} + X_0H_{jk_j}$
with $G'_{jk_j} \in (I_\bbX)_{d+1}$ and $H_{jk_j} \in P_{d}$.
It is clear that $H_{jk_j} \in (I_{\bbX'})_{d}$.
From this we get $F_{jk_j} =
(F'_{jk_j}+X_0^{r_\bbX-d-1}G'_{jk_j})+X_0^{r_\bbX-d}H_{jk_j}$
for all $k_j = 1, \dots, \varkappa_j$. It follows that
$F_{jk_j} \in \langle X_0^{r_\bbX-d},(I_\bbX)_{r_\bbX}\rangle_P$
for all $k_j=1,\dots,\varkappa_j$.
This is a contradiction to our hypothesis, and hence
the claim is completely proved.
\end{proof}

\begin{lemma}\label{lemSec3.4}
Let $A$ be a 0-dimensional local affine $K$-algebra
with maximal ideal $\fm$, let $\fq$ be a $\fm$-primary ideal,
let $R=A/\fq$, and let $\pi: A\rightarrow R$ be the
canonical epimorphism. Let $g\in A$ be an element such that
$\pi(g)\in \Ann_R(\pi(\fm))$ is a non-zero socle element of~$R$,
and suppose $h\in \Ann_A(\fq)$ and $gh\ne0$.
\begin{enumerate}
  \item[(a)] We have $gh \in \Ann_A(\fm)$ and
  $\langle0\rangle:_{\langle g\rangle}\langle h\rangle
  \subseteq \fq$.
  \item[(b)] Every element $f\in A$ with
  $\pi(f)\in \langle \pi(g)\rangle_R\setminus\{0\}$
  satisfies $fh \ne 0$.
  \item[(c)] Let $g_1,...,g_r\in A\setminus\{0\}$.
  If the set $\{\pi(g_1),...,\pi(g_r)\}\subseteq \langle\pi(g)\rangle_R$
  is $K$-linearly independent, then the set $\{g_1h,...,g_rh\}$
  is $K$-linearly independent.
\end{enumerate}
\end{lemma}
\begin{proof}
For (a), let $a\in\fm$ be a non-zero element.
In $R$ we have $\pi(a)\in\pi(\fm)$, and so we get
$\pi(ag)=\pi(a)\pi(g) =0$ or $ag\in\fq$.
It follows that $agh =0$. Hence $gh\in\Ann_A(\fm)$.
Moreover, for
$f\in\langle0\rangle:_{\langle g\rangle}\langle h\rangle$
we have $f=gf'$ for some $f'\in A$ and $gf'h=fh=0$.
Since $gh$ is a socle element of~$A$ and $bgh\ne 0$
for $b\in A\setminus\fm$, we have $\Ann_A(gh)=\fm$.
This implies $f'\in\fm$. Thus $f=gf' \in\fq$.

To prove (b), we consider an element $f\in A$
with $\pi(f)\in \langle \pi(g)\rangle_R\setminus\{0\}$.
Writing $\pi(f)=\pi(g)\pi(f')$ for some $f'\in A\setminus\{0\}$,
we see that $f'\notin \fm$ is a unit and $f= gf'+f''$
with $f''\in\fq$. So, we obtain
$fh = f'gh + f''h = f'gh \ne 0$.

Next, we prove (c).
Suppose that there are $a_1,...,a_r\in K$ such that
$a_1g_1h+\cdots+ a_rg_rh =(a_1g_1+\cdots+a_rg_r)h =0$.
Since $\pi(a_1g_1+\cdots+a_rg_r)\in\langle\pi(g)\rangle_R$,
it follows from (b) that
$\pi(a_1g_1+\cdots+a_rg_r)=a_1\pi(g_1)+\cdots+a_r\pi(g_r) =0$.
By assumption, we get $a_1=\cdots=a_r=0$.
\end{proof}

The first main result of this section is the following characterization
of the Cayley-Bacharach property, which is a generalization
of results for finite sets of $K$-rational points
or for the case that $K$ is an algebraically closed field
found in~\cite[Theorem~4.6]{GKR} and~\cite[Theorem~4.1]{KR1995}.
For $i\ge 0$ we write $F_p$ for the image in~$\calO_{\bbW,p}$
of $F\in P_i$ under the composition map
$P_i \rightarrow (R_\bbW)_i \rightarrow
\prod_{p \in \Supp(\bbW)} \calO_{\bbW,p}
\rightarrow \calO_{\bbW,p}$.
Notice that $F\in I_\bbW$ if and only if $F_p=0$ for all
$p\in \Supp(\bbW)$ (cf.~\cite[Lemma~1.1]{Kre}).

\begin{thm}\label{thmSec3.4}
Let $\bbW \subseteq \bbP^n_K$ be a $0$-dimensional
arithmetically Gorenstein scheme, let $\bbX$
be a subscheme of $\bbW$, let $\bbY$ be the residual
scheme of~$\bbX$ in~$\bbW$, and let $0 \le d \le r_{\bbX}-1$.
Then the following statements are equivalent.

\begin{enumerate}
  \item[(a)] The scheme $\bbX$ has CBP($d$).

  \item[(b)]
  Every subscheme $\bbY'\subseteq \bbW$ containing $\bbY$
  as a maximal $p_j$-subscheme, where $p_j\in\Supp(\bbX)$,
  satisfies $\HF_{I_{\bbY/\bbY'}}(r_\bbW-d-1)>0$.

  \item[(c)] We have $I_\bbW : (I_{\bbY})_{r_\bbW-d-1} = I_\bbX$.

  \item[(d)] We have
  $(I_\bbW)_{r_\bbW-1}:(I_\bbY)_{r_\bbW-d-1}=(I_\bbX)_{d}$.

  \item[(e)] For all $p_j\in\Supp(\bbX)$ and for every
  maximal $p_j$-subscheme $\bbX'\subseteq\bbX$ with standard
  set of separators $\{F_{j1},...,F_{j\varkappa_j}\}$
  there exists a homogeneous element
  $H_j\in(I_{\bbY})_{r_\bbW-d-1}$ such that
  $H_j\cdot\langle F_{j1},...,F_{j\varkappa_j}\rangle_K
  \nsubseteq I_\bbW$.
\end{enumerate}
\end{thm}

\begin{proof}
First we prove the implication (a)$\Rightarrow$(b).
Let $p_j \in \Supp(\bbX)$, let $\varkappa_j =\dim_K\kappa(p_j)$,
let $\bbY'\subseteq \bbW$ be a subscheme containing $\bbY$
as a maximal $p_j$-subscheme, and
let $\bbX'$ be the residual scheme of~$\bbY'$ in~$\bbW$.
Proposition~\ref{propSec2.10} shows that
$\bbX'$ is exactly a maximal $p_j$-subscheme of~$\bbX$
of degree $\deg(\bbX')=\deg(\bbX)-\varkappa_j$.
By Proposition~\ref{propSec2.5}, we observe that
$r_{\bbX'}+\alpha_{\bbY'/\bbW}=r_\bbW=r_\bbX+\alpha_{\bbY/\bbW}$,
and $\HF_{I_{\bbY/\bbW}}(i)=\deg(\bbX)-\HF_{\bbX}(r_\bbW-i-1)$
and $\HF_{I_{\bbY'/\bbW}}(i)
=\deg(\bbX')-\HF_{\bbX'}(r_\bbW-i-1)$
for all $i\in \bbZ$. So, for all $i\in \bbZ$, we have
$\HF_{I_{\bbY/\bbY'}}(i)=
\varkappa_j - \HF_{I_{\bbX'/\bbX}}(r_\bbW-i-1)$.
According to Proposition~\ref{propSec3.2}, the Hilbert function
of $I_{\bbX'/\bbX}$ satisfies
$\HF_{I_{\bbX'/\bbX}}(d) < \varkappa_j$.
Consequently, we get
$\HF_{I_{\bbY/\bbY'}}(r_\bbW-d-1) =\varkappa_j -
\HF_{I_{\bbX'/\bbX}}(d)> 0$, as wanted.

Now we prove the implication (b)$\Rightarrow$(c).
Clearly, $I_\bbX \subseteq I_\bbW:(I_\bbY)_{r_\bbW-d-1}$.
Suppose for a contradiction that
$F \in I_\bbW : (I_\bbY)_{r_\bbW-d-1}$
and $F \notin I_\bbX$.
There is a point $p_j\in \Supp(\bbX)$
such that $F_{p_j} \ne 0$. By~\cite[Lemma~4.5.9(a)]{KR2016}
there is $a_j \in \calO_{\bbX,p_j}$ such that
$a_j\cdot F_{p_j}$ is a socle element of~$\calO_{\bbX,p_j}$.
This socle element defines a maximal $p_j$-subscheme
$\bbX'$ of $\bbX$ by~\cite[Proposition~4.2]{KL}.
Then the residual scheme $\bbY'$ of~$\bbX'$ in~$\bbW$
satisfies $\HF_{I_{\bbY/\bbY'}}(r_\bbW-d-1)>0$ by
Proposition~\ref{propSec2.10} and (b).
On the other hand, letting $G \in (I_\bbY)_{r_\bbW-d-1}$,
then $FG \in I_\bbW$ and $G\cdot I_\bbX\subseteq I_\bbW$.
Since $I_\bbW$ is saturated, we have
$G\cdot \langle F,I_\bbX \rangle^{\rm sat} \subseteq I_\bbW$.
So, $G\cdot I_{\bbX'}\subseteq I_\bbW$
or $G\in(I_{\bbY'})_{r_\bbW-d-1}$, as
$I_{\bbX'} \subseteq \langle F,I_\bbX \rangle^{\rm sat}$.
Hence we get $\HF_{I_{\bbY/\bbY'}}(r_\bbW-d-1)=0$, a contradiction.

Moreover, the implication (c)$\Rightarrow$(d) is clear.
Next, we prove the implication (d)$\Rightarrow$(e).
Let $\bbX'\subseteq\bbX$ be a maximal $p_j$-subscheme with
set of minimal separators $\{F^*_{j1},\dots,F^*_{j\varkappa_j}\}$.
If there exists some index $k_j\in\{1,\dots,\varkappa_j\}$
such that $\deg(F^*_{jk_j}) \le d$,
then $G_{jk_j} = X_0^{d-\deg(F^*_{jk_j})}F^*_{jk_j}
\notin (I_\bbX)_d$, and so claim (d) implies
$G_{jk_j} \notin (I_\bbW)_{r_\bbW-1}:(I_\bbY)_{r_\bbW-d-1}$.
Let $H_j \in (I_\bbY)_{r_\bbW-d-1}\setminus\{0\}$
be such that $G_{jk_j}H_j \notin (I_\bbW)_{r_\bbW-1}$.
Since $X_0$ is a non-zerodivisor for $R_\bbW$, we have
$F_{jk_j}H_j=X_0^{r_\bbX-d}G_{jk_j}H_j\notin I_\bbW$.
In case $\deg(F^*_{jk_j})>d$ for all $k_j=1,\dots,\varkappa_j$,
we see that $F_{jk_j} \notin \langle\, X_0^{r_{\bbX}-d},
(I_{\bbX})_{r_{\bbX}} \,\rangle_P$.
Let $1\le \delta \le r_\bbX-d$ be the smallest number such that
$\langle F_{j1},\dots,F_{jk_j}\rangle_K \nsubseteq
\langle X_0^\delta,(I_\bbX)_{r_\bbX}\rangle_P$.
Write $F_{jk_j} = F'_{jk_j} + X_0^{\delta-1}G_{jk_j}$
with $F'_{jk_j} \in (I_{\bbX})_{r_\bbX}$ and
$G_{jk_j} \in P_{r_\bbX-\delta+1}$.
Then Lemma~\ref{lemSec3.3} yields that 
$G_{jk_j} \notin (I_\bbW)_{r_\bbW} :
(I_\bbY)_{\alpha_{\bbY/\bbW}+\delta-1}$
for some $k_j \in \{1, \dots, \varkappa_j\}$.
So, there is an element
$\widetilde{H}_j\in(I_\bbY)_{\alpha_{\bbY/\bbW}+\delta-1}$
such that $G_{jk_j}\widetilde{H}_j \notin (I_\bbW)_{r_\bbW}$.
Set $H_j=X_0^{r_\bbX-d-\delta}\widetilde{H}_j
\in (I_\bbY)_{r_\bbW-d-1}$.
Since $F'_{jk_j}\widetilde{H}_j\in I_\bbW$, we get
$F_{jk_j}H_j \notin I_\bbW$.

Finally, we prove the implication (e)$\Rightarrow$(a).
For a contradiction, assume that $\bbX$ does not have CBP($d$),
and let $\bbX'\subseteq\bbX$ be a maximal $p_j$-subscheme
such that its minimal separators satisfies
$\deg(F^*_{jk_j}) \le d$ for all $k_j=1,\dots,\varkappa_j$.
Set $G_{jk_j} = X_0^{d-\deg(F^*_{jk_j})}F^*_{jk_j}$
for $k_j=1,\dots,\varkappa_j$.
By (e) there exists $H_j \in (I_\bbY)_{r_\bbW-d-1}$
and some $k_j\in\{1,...,\varkappa_j\}$ such that
$G_{jk_j}H_j\notin (I_\bbW)_{r_\bbW-1}$. W.l.o.g. assume that
$G_{j1}H_j\notin (I_\bbW)_{r_\bbW-1}$.
Notice that, as sets, $\Supp(\bbW)=\Supp(\bbX)\cup\Supp(\bbY)$.
In~$\calO_{\bbW,p_j}$, we have $(G_{j1}H_j)_{p_j}\ne0$ and
$(G_{j1}H_j)_p=0$ for any $p\in \Supp(\bbW)\setminus\{p_j\}$.
Also, by writing $\calO_{\bbX,p_j} = \calO_{\bbW,p_j}/\fq_j$
for some ideal $\fq_j$ of~$\calO_{\bbW,p_j}$, we have
$\fq_j \cdot (H_j)_{p_j} = \langle0\rangle$ in~$\calO_{\bbW,p}$
and $(G_{j1})_{p_j} \in \calO_{\bbX,p_j}$ is a socle element
with $(G_{jk_j})_{p_j} \in
\langle (G_{j1})_{p_j}\rangle_{\calO_{\bbX,p_j}} \setminus\{0\}$
for all $k_j=1,...,\varkappa_j$.
In particular, by the definition of minimal separators,
the set $\{(G_{j1})_{p_j},...,(G_{j\varkappa_j})_{p_j}\}$
is $K$-linearly independent.
Thus Lemma~\ref{lemSec3.4} yields that
$(G_{j1}H_j)_{p_j}$ is a socle element of~$\calO_{\bbW,p_j}$
and $\{(G_{j1}H_j)_{p_j}, \dots, (G_{j\varkappa_j}H_j)_{p_j}\}
\subseteq \calO_{\bbW,p_j}$
is $K$-linearly independent.
Set $J :=
\langle G_{jk}H_j+I_\bbW \mid 1\le k\le\varkappa_j\rangle_{R_\bbW}$.
Obviously, we have
$$
\dim_K J_{r_\bbW-1+i} \ge \dim_K\langle 
(G_{j1}H_j)_{p_j},\dots,(G_{j\varkappa_j}H_j)_{p_j}
\rangle_K = \varkappa_j
$$
for all $i\ge 0$.
Furthermore, using \cite[Lemma~2.8]{KLL} we write
$$
X_iG_{jl}+I_\bbX = {\textstyle\sum\limits_{k_j=1}^{\varkappa_j}}
c_{jk_jl}X_0G_{jk_j}+I_\bbX
$$
for some $c_{j1l},\dots,c_{j\varkappa_jl} \in K$,
where $0 \le i \le n$ and $1 \le l \le \varkappa_j$.
Then we get
$$
X_iG_{jl}H_j+I_\bbW ={\textstyle\sum\limits_{k_j=1}^{\varkappa_j}}
c_{jk_jl}X_0G_{jk_j}H_j + I_\bbW,
$$
and subsequently
$\dim_KJ_{r_\bbW-1+i} = \varkappa_j$ for all $i\ge0$.
Consequently, the homogeneous ideal $J$ defines a maximal
$p_j$-subscheme $\bbW'\subseteq\bbW$ such that
$\dim_K(I_{\bbW'/\bbW})_{r_\bbW-1}=\varkappa_j$.
Therefore Proposition~\ref{propSec3.2} implies that
$\bbW$ is not a Cayley-Bacharach scheme.
But $\bbW$ is arithmetically Gorenstein, and so it is a
Cayley-Bacharach scheme by \cite[Proposition~4.8]{KL},
and this is a contradiction.
\end{proof}

Let us apply Theorem~\ref{thmSec3.4} to a concrete case.

\begin{example} \label{examS3.6}
Let $K$ be a field with ${\rm char}(K) \ne 2, 3$, and
let $\bbW \subseteq \bbP^2_K$ be the $0$-dimensional
complete intersection defined by $I_\bbW=\langle F,G\rangle$,
where $F=X_1(X_1-2X_0)(X_1+2X_0)$ and
$G=(X_2-X_0)(X_1^2+X_2^2-4X_0^2)$.
Then $\bbW$ has degree $9$ and the support of~$\bbW$ is
$\Supp(\bbW) = \{p_1, \dots, p_7\}$ with
$p_1=(1:0:1)$, $p_2=(1:0:2)$, $p_3=(1:0:-2)$, $p_4=(1:2:1)$,
$p_5=(1:2:0)$, $p_6=(1:-2:1)$, and $p_7=(1:-2:0)$.
A homogeneous primary decomposition of~$I_\bbW$ is
$I_\bbW = I_1 \cap \cdots \cap I_7$,
where $I_i$ is the homogeneous prime ideal
corresponding to $p_i$ for $i \ne 5, 7$,
$I_5 = \langle X_1-2X_0, X_2^2 \rangle$,
and  $I_7 = \langle X_1+2X_0, X_2^2 \rangle$.
So, the scheme $\bbW$ is arithmetically Gorenstein,
but not reduced at $p_5$ and $p_7$.

Now we consider the $0$-dimensional subscheme $\bbX$
of~$\bbW$ defined by the ideal
$I_{\bbX} = I_1 \cap I_3 \cap I_4 \cap I_5 \subseteq P$.
Then $\deg(\bbX)=5$ and $\bbX$ is not reduced.
The residual scheme of~$\bbX$ in~$\bbW$ is denoted by~$\bbY$.
It is easy to see that $\bbX$ and
$\bbY$ are geometrically linked.
We have $r_{\bbW} = 4$ and $r_{\bbX} =
\alpha_{\bbX/\bbW} = r_{\bbY} = \alpha_{\bbY/\bbW} = 2$.
In this case there is a homogeneous polynomial
$H \in (I_{\bbY})_{2}$ such that
its image in $R_{\bbX}$ is a non-zerodivisor, for instance,
$H = X_{0}^{2} + X_{0}X_{1} + \tfrac{1}{4} X_{1}^{2}
- \tfrac{1}{2} X_{0}X_{2} - \tfrac{1}{4} X_{1}X_{2}$.
This polynomial satisfies the condition
(e) in~Theorem~\ref{thmSec3.4}. Therefore
$\bbX$ is a Cayley-Bacharach scheme.
\end{example}

The above example shows that, setting
$I_{\bbY,\bbX} := (I_\bbY+I_\bbX)/I_\bbX$, the condition
$\Ann_{R_\bbX}((I_{\bbY,\bbX})_{r_\bbW-d-1})=\langle0\rangle$
is a sufficient condition for $\bbX$ having CBP($d$)
in this case. In general case, this is also true.
Indeed, if
$\Ann_{R_\bbX}((I_{\bbY,\bbX})_{r_\bbW-d-1})=\langle0\rangle$
then for each maximal $p_j$-subscheme $\bbX'\subseteq\bbX$
with standard set of separators $\{F_{j1},...,F_{j\varkappa_j}\}$
there is a non-zero homogeneous element
$H_j\in (I_{\bbY,\bbX})_{r_\bbW-d-1}$
such that $(H_j)_{p_j}\in \calO_{\bbX,p_j}\setminus\fm_{\bbX,p_j}$,
and so $(H_j)_{p_j}\notin \fm_{\bbW,p_j}$ and
$(H_jF_{jk_j})_{p_j}\ne 0$ in $\calO_{\bbW,p_j}$.
This means that $H_jF_{jk_j} \notin I_\bbW$.
Subsequently, the condition (e) of Theorem~\ref{thmSec3.4}
is satisfied, and hence $\bbX$ has CBP($d$).

However, the above condition is not a necessary condition for
$\bbX$ having CBP($d$), as our next example shows.

\begin{example}\label{examS3.7}
Let $K$ be a field with ${\rm char}(K) \ne 2, 3$,
let $\bbW \subseteq \bbP^2_K$ be the $0$-dimensional
complete intersection given in Example~\ref{examS3.6},
and let $\bbX'$ be the set of points in~$\bbW$
with its homogeneous vanishing ideal
$I_{\bbX'} = I_1 \cap I_3 \cap I_4 \cap I'_5$,
where $I'_5 = \langle X_1-2X_0, X_2 \rangle$ is the
homogeneous prime ideal corresponding to $p_5$.
Then the residual scheme $\bbY'$ of~$\bbX'$ in~$\bbW$
has the homogeneous vanishing ideal
$I_{\bbY'} = I_2 \cap I'_5 \cap I_6 \cap I_7$.
It is clear that $r_{\bbX'} = \alpha_{\bbX'/\bbW}
= r_{\bbY'} = \alpha_{\bbY'/\bbW} = 2$ and
$$
I_{\bbY'} = \langle
X_{0}^{2} - \tfrac{1}{4} X_{1}^{2}
- \tfrac{1}{2} X_{0}X_{2} - \tfrac{1}{4} X_{1}X_{2},
X_0X_1X_2 + \tfrac{1}{2} X_1^2X_2,
X_0X_2^2 + \tfrac{1}{4}X_1X_2^2-\tfrac{1}{2}X_2^3
\rangle.
$$
In this case it is not difficult to verify that
the scheme $\bbX'$ is a complete intersection,
and hence it is a Cayley-Bacharach scheme.
However, there is no element $H$ in~$(I_{\bbY'})_2$
such that $H_{p_5} \ne 0$ in~$\calO_{\bbX',p_5}$.
Hence the condition
$\Ann_{R_{\bbX'}}((I_{\bbY',\bbX'})_{r_\bbW-r_{\bbX'}})
=\langle0\rangle$ is not satisfied,
even when $\bbX'$ is a Cayley-Bacharach scheme.
Moreover, we see that the element $F_5 = X_1^2 - 2X_1X_2$
is a minimal separator of~$\bbX'\setminus\{p_5\}$
in~$\bbX'$ and $(F_5H_5)_{p_5}$ is a socle element
of~$\calO_{\bbW,p_5}$, where
$H_5 = X_{0}^{2} - \tfrac{1}{4} X_{1}^{2}
- \tfrac{1}{2} X_{0}X_{2}-\tfrac{1}{4} X_{1}X_{2}
\in (I_{\bbY'})_2$.
\end{example}

It is interesting to examine the natural question
whether the condition that $\bbX$ has CBP($d$) is
equivalent to
$\Ann_{R_\bbX}((I_{\bbY,\bbX})_{r_\bbW-d-1})=\langle0\rangle$.
When the schemes $\bbW$, $\bbX$ and $\bbY$ are
finite sets of $K$-rational points in~$\bbP^n_K$
and $\bbW$ is a complete intersection,
this question has an affirmative answer
as was shown in~\cite[Theorem~4.6]{GKR}.
In our more general setting, this result can be
generalized as follows.

\begin{thm} \label{thmSec3.8}
\!\!\! Let $\bbX$\! and $\bbY$\! be geometrically linked by a
0-dimensional arithmetically Gorenstein scheme $\bbW$,
and let $I_{\bbY,\bbX}=(I_\bbY+I_\bbX)/I_\bbX$.
Then the scheme $\bbX$ has CBP($d$) if and only if we have
$\Ann_{R_\bbX}((I_{\bbY,\bbX})_{r_\bbW-d-1})=\langle0\rangle$.
\end{thm}

\begin{proof}
According to the argument before Example~\ref{examS3.7}, it
suffices to show that
$\Ann_{R_\bbX}((I_{\bbY,\bbX})_{r_\bbW-d-1})=\langle0\rangle$
if $\bbX$ has CBP($d$).
To this end, let $\bbX' \subseteq \bbX$ be
a maximal $p_j$-subscheme with standard set of separators
$\{F_{j1}, \dots, F_{j\varkappa_j}\} \subseteq P_{r_{\bbX}}$.
Since $\bbX$ has CBP($d$), Theorem~\ref{thmSec3.4}
yields that there is an element $H_j \in (I_\bbY)_{r_\bbW-d-1}$
such that $H_j\cdot\langle F_{j1},...,F_{j\varkappa_j}
\rangle_K\nsubseteq I_\bbW$. W.l.o.g. we assume that
$H_jF_{j1} \notin I_\bbW$.
Since $\bbX$ and $\bbY$ are geometrically linked,
$(F_{j1})_{p_j}$ is a socle element
in~$\calO_{\bbW,p_j}=\calO_{\bbX,p_j}$.
Since $(H_jF_{j1})_p = 0$ in~$\calO_{\bbW,p}$
for every $p \in \Supp(\bbW)\setminus\{p_j\}$
and $(H_jF_{j1})_{p_j}\ne0$, we get
$(H_j)_{p_j} \notin \fm_{\bbX,p_j}$.
Consequently, for each point~$p_j$ of~$\Supp(\bbX)$,
we can find an element $H_j \in (I_\bbY)_{r_\bbW-d-1}$ such that
$(H_j)_{p_j}$ is a unit of $\calO_{\bbX,p_j}$.
By~\cite[Lemma~1.1]{Kre}, this condition
is exactly the right condition to have
$\Ann_{R_\bbX}((I_{\bbY,\bbX})_{r_\bbW-d-1})=\langle0\rangle$.
\end{proof}

\begin{remark}
Let $\bbX\subseteq \bbP^n_K$ be a $0$-dimensional scheme.
\begin{enumerate}
  \item[(a)]
  If $\bbX$ is reduced and has $K$-rational support,
  then there is a complete intersection consisting of distinct
  $K$-rational points $\bbW$ containing $\bbX$ such that
  $\bbX$ and its residue scheme by~$\bbW$ are geometrically linked
  (see, e.g., \cite[Remark~4.11]{GKR}).

  \item[(b)]
  If $\calO_{\bbX,p_j}$ is not a Gorenstein local ring
  for some point $p_j\in\Supp(\bbX)$, then
  there is no $0$-dimensional arithmetically
  Gorenstein scheme $\bbW \subseteq \bbP^n_K$ containing
  $\bbX$ such that $\bbX$ and its residual scheme in~$\bbW$
  are geometrically linked.
\end{enumerate}
\end{remark}

We end this section with the following immediate consequence
of Theorem~\ref{thmSec3.4}. This result allows us to
check whether $\bbX$ has CBP($d$)
by using a truncated Gr\"obner basis calculation
(cf.~\cite[Section~4.5]{KR2005}).
For the case of sets of distinct $K$-rational points
and $d=r_\bbX-1$ see also \cite[Corollary~4.10]{GKR}.

\begin{corollary} \label{corSec3.7}
In the setting of Theorem~\ref{thmSec3.4}, the scheme
$\bbX$ has CBP($d$) if and only if
$\HF_{P/(I_\bbW:(I_\bbW:I_\bbX)_{r_\bbW-d-1})}(d)
= \HF_\bbX(d)$.
\end{corollary}

\medskip \bigbreak
\section{Bound the Hilbert Function of the Dedekind Different}
\label{Section4}

In this section, we let $\bbX\subseteq \bbP^n_K$ be a
0-dimensional scheme and we let $0 \le d < r_{\bbX}$.
The aim of this section is to characterize the Cayley-Bacharach
property using the canonical module of $R_\bbX$, and apply
these results to bound the Hilbert function
and determine the regularity index of the Dedekind different
of~$\bbX$ under some additional hypotheses.

Recall that the graded $R_\bbX$-module
$\omega_{R_\bbX} = \underline{\Hom}_{K[x_0]}(R_\bbX,K[x_0])(-1)$
is called the {\it canonical module} of~$R_\bbX$.
Its $R_\bbX$-module structure is defined by
$(f\cdot\varphi)(g) = \varphi(fg)$ for all $f,g\in R_\bbX$
and $\varphi\in \omega_{R_\bbX}$.
It is also a finitely generated graded $R_\bbX$-module and
$$
\HF_{\omega_{R_\bbX}}(i) = \deg(\bbX) - \HF_{\bbX}(-i)
\quad \mbox{ for all } i\in\bbZ.
$$
The following two lemmas give us some more information
about this module.

\begin{lemma}\label{lemSec4.1}
For every homogeneous element
$\varphi \in (\omega_{R_\bbX})_{-d}$ its restriction
$\overline{\varphi}=\varphi|_{(R_\bbX)_{d+1}}:
(R_\bbX)_{d+1}\rightarrow K$
is a $K$-linear map such that
$\overline{\varphi}(x_0(R_\bbX)_{d})\!=\!\langle 0\rangle$.
Conversely, if $\overline{\varphi}: (R_\bbX)_{d+1} \rightarrow K$
is a $K$-linear map such that
$\overline{\varphi}(x_0(R_\bbX)_{d})=\langle 0\rangle$,
then there exists a homogeneous element
$\varphi \in (\omega_{R_\bbX})_{-d}$ such that
$\varphi|_{(R_\bbX)_{d+1}}=\overline{\varphi}$.
\end{lemma}
\begin{proof}
Clearly, for every homogeneous element
$\varphi \in (\omega_{R_\bbX})_{-d}$ its restriction
$\overline{\varphi}=\varphi|_{(R_\bbX)_{d+1}}$ is
a $K$-linear map. Also, we have
$$
\overline{\varphi}(x_0(R_\bbX)_{d}) =
\varphi(x_0(R_\bbX)_{d}) = x_0 \varphi((R_\bbX)_{d})
\subseteq x_0(K[x_0])_{-1}=\langle 0\rangle.
$$
Now let $\overline{\varphi}: (R_\bbX)_{d+1} \rightarrow K$
is a $K$-linear map such that
$\overline{\varphi}(x_0(R_\bbX)_{d})=\langle 0\rangle$.
Let $h_i = \HF_\bbX(i)-\HF_\bbX(i-1)$ for $i\in \bbN$.
Note that $(R_\bbX)_i=x_0^{i-r_\bbX}(R_\bbX)_{r_\bbX}$
and $h_i=0$ for all $i>r_\bbX$.
To define an element $\varphi \in (\omega_{R_\bbX})_{-d}$
with the desired properties, we start taking a $K$-basis
$g_1,...,g_{\scriptscriptstyle\sum\limits_{0\le k\le d+1}h_k}$
of $(R_\bbX)_{d+1}$.
For $i=d+2,...,r_\bbX$, we choose
$g_{\scriptscriptstyle\sum\limits_{0\le k<i}h_k+1},...,
g_{\scriptscriptstyle\sum\limits_{0\le k\le i}h_k}$
such that the set
$$
\Big\{\,
x_0^{i-d-1}g_1,...,
x_0^{i-d-1}g_{\scriptscriptstyle\sum\limits_{0\le k\le d+1}h_k},
..., g_{\scriptscriptstyle\sum\limits_{0\le k<i}h_k+1},...,
g_{\scriptscriptstyle\sum\limits_{0\le k\le i}h_k}
\,\Big\}
$$
forms a $K$-basis of $(R_\bbX)_i$. Then we get
$$
(R_\bbX)_i = \langle x_0^{i-d-1}g_1,...,
x_0^{i-d-1}g_{\scriptscriptstyle\sum\limits_{0\le k\le d+1}h_k},
..., x_0^{i-r_\bbX}g_{\scriptscriptstyle\sum\limits_{0\le k<r_\bbX}h_k+1},...,
x_0^{i-r_\bbX}g_{\scriptscriptstyle\sum\limits_{0\le k\le r_\bbX}h_k}
\rangle_K
$$
for all $i\ge r_\bbX$.
Let $\varphi: R_\bbX \rightarrow K[x_0]$ be the homogeneous
$K$-linear map of degree~$-d$ defined as:
for $f\in R_i$ with $i\le d$ we let $\varphi(f)=0$,
and for $f\in R_i$ with $i\ge d+1$ we write
$$
f = \sum_{1\le j \le \scriptscriptstyle
\sum\limits_{0\le k\le d+1}h_k}a_jx_0^{i-d-1}g_j + \cdots +
\sum_{\scriptscriptstyle
\sum\limits_{0\le k<r_\bbX}h_k+1 \le j \le
\scriptscriptstyle\sum\limits_{0\le k\le r_\bbX}
h_k }a_jx_0^{i-r_\bbX}g_j
$$
and let $\varphi(f)=
\sum_{1\le j \le \scriptscriptstyle
\sum\limits_{0\le k\le d+1}h_k}a_jx_0^{i-d-1}
\overline{\varphi}(g_j)$. The condition
$\overline{\varphi}(x_0(R_\bbX)_{d})=\langle 0\rangle$
implies that the map $\varphi$ is $K[x_0]$-linear.
Hence $\varphi\in (\omega_{R_\bbX})_{-d}$
is the desired element that we wanted to construct.
\end{proof}

\begin{lemma} \label{lemSec4.2}
The canonical module $\omega_{R_\bbX}$ satisfies
$\Ann_{R_\bbX}((\omega_{R_\bbX})_{-d})=\langle 0\rangle$
if and only if for every $p_j\in\Supp(\bbX)$ and
for every maximal $p_j$-subscheme
$\bbX' \subseteq \bbX$ there exists a homogeneous element
$\varphi \in (\omega_{R_\bbX})_{-d}$ such that
$I_{\bbX'/\bbX}\cdot \varphi \ne \langle 0\rangle$.
\end{lemma}
\begin{proof}
We need only to prove that if for every $p_j\in\Supp(\bbX)$
and for every maximal $p_j$-subscheme
$\bbX' \subseteq \bbX$ there exists a homogeneous element
$\varphi \in (\omega_{R_\bbX})_{-d}$ such that
$I_{\bbX'/\bbX}\cdot \varphi \ne \langle 0\rangle$ then
$\Ann_{R_\bbX}((\omega_{R_\bbX})_{-d})=\langle 0\rangle$.
Suppose for a contradiction that
$f\cdot(\omega_{R_\bbX})_{-d} = \langle 0\rangle$
for some $f\in (R_\bbX)_i\setminus\{0\}$ with $i\ge 0$.
Since $f\ne 0$, we may assume the germ $f_{p_j}\ne 0$
for some $j\in\{1,...,s\}$. In the local ring
$\calO_{\bbX,p_j}$ we find an element $a \in \calO_{\bbX,p_j}$
such that $s_j=af_{p_j}$ is a socle element of $\calO_{\bbX,p_j}$
(cf.~\cite[Lemma~4.5.9(a)]{KR2016}).
Now let
$g=\tilde{\imath}^{-1}((0,...,0,s_jT_j^{r_\bbX},0,...,0))$ and
$h=\tilde{\imath}^{-1}((0,...,0,aT_j^{r_\bbX},0,...,0))$.
Then $g,h\in (R_\bbX)_{r_\bbX}$ satisfies $x_0^ig=fh$.
Also, the ideal $\langle g\rangle$ defines a maximal
$p_j$-subscheme $\bbX'$ of~$\bbX$, that is, we have
$I_{\bbX'/\bbX} = \langle g\rangle^{\rm sat}$.
Thus there is $\varphi \in (\omega_{R_\bbX})_{-d}$ such that
$\langle g\rangle^{\rm sat}\cdot \varphi \ne \langle 0\rangle$,
in particularly, $g\cdot\varphi\ne 0$. It follow that
$g\cdot\varphi(\widetilde{g}) \ne 0$ for some non-zero
homogeneous element $\widetilde{g}\in R_\bbX$.
Hence we get $0 = (f\cdot\varphi)(h\widetilde{g})
= (fh\cdot\varphi)(\widetilde{g})
= (x_0^ig\cdot\varphi)(\widetilde{g})
= (g\cdot\varphi)(x_0^i\widetilde{g})
= x_0^i (g\cdot\varphi)(\widetilde{g}) \ne 0$,
a contradiction.
\end{proof}

Using the above properties we prove the following
characterization of the Cayley-Bacharach property
in terms of the canonical module.

\begin{proposition}\label{propSec4.3}
Let $\bbX\subseteq \bbP^n_K$ be a 0-dimensional scheme,
and let $0 \le d < r_{\bbX}$.
Then the following conditions are equivalent.
\begin{enumerate}
  \item[(a)] The scheme $\bbX$ has CBP($d$).
  \item[(b)] We have
  $\Ann_{R_\bbX}((\omega_{R_\bbX})_{-d})=\langle0\rangle$.
\end{enumerate}
\end{proposition}
\begin{proof}
Suppose that $\bbX$ has CBP($d$). Let $\bbX'\subseteq \bbX$
be a maximal $p_j$-subscheme with set of minimal separators
$\{f^*_{j1},...,f^*_{j\varkappa_j}\}$.
By Proposition~\ref{propSec3.2}, there exists an index
$k\in\{1,...,\varkappa_j\}$ such that $\rho=\deg(f^*_{jk})\ge d+1$
and $f^*_{jk}\notin x_0(R_\bbX)_{\rho-1}$.
W.l.o.g. assume that $k=1$. So, we can define a $K$-linear map
$\overline{\varphi}_j: (R_\bbX)_\rho\rightarrow K$ such that
$\overline{\varphi}_j(x_0(R_\bbX)_{\rho-1})=\langle 0\rangle$
and $\overline{\varphi}_j(f^*_{j1})\ne 0$.
Using Lemma~\ref{lemSec4.1} we lift this map to obtain
a homogeneous element $\varphi_j \in (\omega_{R_\bbX})_{-\rho+1}$
such that $\varphi_j(f^*_{j1})\ne 0$. Since $x_0$ is a
non-zerodivisor of~$R_\bbX$, it follows that
$x_0^{\rho-d-1}\varphi_j(f^*_{j1})\ne 0$. Especially, we have
$x_0^{\rho-d-1}\cdot\varphi_j \in (\omega_{R_\bbX})_{-d}$ and
$I_{\bbX'/\bbX}\cdot (x_0^{\rho-d-1}\cdot\varphi_j)
\ne \langle 0\rangle$. Hence Lemma~\ref{lemSec4.2} yields
the condition
$\Ann_{R_\bbX}((\omega_{R_\bbX})_{-d})=\langle0\rangle$.

Conversely, assume for a contradiction that
$\bbX$ does not have CBP($d$).
There is a maximal $p_j$-subscheme
$\bbX'\subseteq \bbX$ such that its set of minimal separators
$\{f^*_{j1},...,f^*_{j\varkappa_j}\}$ satisfies
$\deg(f^*_{jk})\le d$ for all $k=1,...,\varkappa_j$.
By Remark~\ref{remSec2.3}(a)-(c), we may assume that,
for $i\ge 0$, the set
$\{x_0^{i-\deg(f^*_{jk_j})}f^*_{jk_j} \,\mid\, \deg(f^*_{jk_j})\le i \}$
is a $K$-basis of~$(I_{\bbX'/\bbX})_i$.
In this case for every $\varphi \in (\omega_{R_\bbX})_{-d}$
we have $\varphi(f^*_{jk})=0$ for all $k=1,...,\varkappa_j$.
We shall show that $f^*_{jk}\cdot \varphi =0$ for all
$k=1,...,\varkappa_j$.
Let $i\ge 0$ and let $h\in R_i\setminus\{0\}$
be a homogeneous element.
If $hf^*_{jk}=0$ then
$(f^*_{jk}\cdot\varphi)(h)=\varphi(hf^*_{jk})=0$.
Suppose that $hf^*_{jk}\ne 0$. Since $hf^*_{jk}\in I_{\bbX'/\bbX}$,
this allows us to write
$hf^*_{jk} = \sum_{l=1}^{\varkappa_j}c_{jl}
x_0^{i+\deg(f^*_{jk})-\deg(f^*_{jl} )}f^*_{jl}$
for some $c_{j1},...,c_{j\varkappa_j}\in K$.
This implies
$(f^*_{jk}\cdot\varphi)(h)
= \varphi(hf^*_{jk})
=\sum_{l=1}^{\varkappa_j}c_{jl}
x_0^{i+\deg(f^*_{jk})-\deg(f^*_{jl} )}\varphi(f^*_{jl})=0$.
Hence we have shown $f^*_{jk}\cdot \varphi =0$ for
all $k=1,...,\varkappa_j$. In addition, we have
$I_{\bbX'/\bbX} = \langle f^*_{j1},...,f^*_{j\varkappa_j}\rangle$.
It follows that $I_{\bbX'/\bbX}\cdot \varphi = \langle 0\rangle$
for any homogeneous element $\varphi \in (\omega_{R_\bbX})_{-d}$.
Therefore we get
$\Ann_{R_\bbX}((\omega_{R_\bbX})_{-d})\ne \langle0\rangle$,
a contradiction.
\end{proof}

As a consequence of the proposition, we get the following property.
Here we recall that a 0-dimensional scheme $\bbX$
is called {\it locally Gorenstein} if the local ring
$\calO_{\bbX,p_j}$ is a Gorenstein ring for every point
$p_j\in\Supp(\bbX)$.

\begin{corollary}\label{corSec4.4}
Let $K$ be an infinite field,
let $\bbX\subseteq \bbP^n_K$ be a 0-dimensional locally
Gorenstein scheme, and let $0 \le d < r_{\bbX}$.
Then $\bbX$ has CBP($d$) if and only if there exists
an element $\varphi\in(\omega_{R_\bbX})_{-d}$ such that
$\Ann_{R_\bbX}(\varphi) = \langle0\rangle$.
\end{corollary}
\begin{proof}
Since $\bbX$ is locally Gorenstein, there is for each point
$p_j\in\Supp(\bbX)$ a uniquely maximal $p_j$-subscheme $\bbX'_j$
of~$\bbX$. So, the condition (b) of Proposition~\ref{propSec4.3}
is equivalent to the condition that for each $j\in\{1,...,s\}$
there exists $\varphi_j \in(\omega_{R_\bbX})_{-d}$ such that
$I_{\bbX'_j/\bbX}\cdot \varphi_j \ne \langle 0\rangle$.
This is in turn equivalent to that there exists
$\varphi\in(\omega_{R_\bbX})_{-d}$ such that
$I_{\bbX'_j/\bbX}\cdot \varphi \ne \langle 0\rangle$
for $j=1,...,s$, since the base field $K$ is infinite,
and this condition is exactly the right condition to make
$\Ann_{R_\bbX}(\varphi) = \langle0\rangle$.
\end{proof}

\begin{remark}
This corollary is a generalization of a result
for the case $d=r_\bbX-1$ found in~\cite[Proposition~4.12]{KL}.
Moreover, the hypothesis in the corollary that $K$ is infinite
is necessary (cf.~\cite[Example~4.14]{KL}).
\end{remark}

Now let us apply the above results to look at the
Hilbert function of the Dedekind different of~$\bbX$.
For this purpose, we assume, in what follows, that
$\bbX$ is locally Gorenstein, and we let $L_0 =K[x_0,x_0^{-1}]$.
The homogeneous ring of quotients of $R_\bbX$ is
$Q^h(R_\bbX)\cong\prod_{j=1}^s\calO_{\bbX,p_j}[T_j,T_j^{-1}]$.
According to \cite[Proposition~3.3]{KL}, the graded algebra
$Q^h(R_\bbX)/L_0$ has a homogeneous trace map $\sigma$
of degree zero, i.e.,
$\sigma\in (\underline{\Hom}_{L_0}(Q^h(R_\bbX),L_0))_0$
satisfies $\underline{\Hom}_{L_0}(Q^h(R_\bbX),L_0)
= Q^h(R_\bbX)\cdot\sigma$.
Thus there is an injective homomorphism of graded $R_\bbX$-modules
\[
\begin{aligned}
\Phi: \omega_{R_\bbX}(1) &
\lhook\joinrel\longrightarrow \underline{\Hom}_{L_0}(Q^h(R_\bbX),L_0)
= Q^h(R_\bbX)\cdot\sigma \xrightarrow{\sim} Q^h(R_\bbX)\\
\varphi &\longmapsto \varphi\otimes \mathrm{id}_{L_0}
\end{aligned}
\]
The image of~$\Phi$ is a homogeneous fractional
$R_\bbX$-ideal~$\fC^{\sigma}_\bbX$ of~$Q^h(R_\bbX)$.
It is also a finitely generated graded $R_\bbX$-module and
$$
\HF_{\fC^{\sigma}_{\bbX}}(i) = \deg(\bbX) - \HF_{\bbX}(-i-1)
\quad \mbox{ for all } i\in\bbZ.
$$

\begin{definition}
The $R$-module $\fC^{\sigma}_{\bbX}$
is called the \textbf{Dedekind complementary module} of~$\bbX$
with respect to~$\sigma$. Its inverse,
$$
\delta^{\sigma}_\bbX = (\fC^{\sigma}_\bbX)^{-1}
= \{\, f \in Q^h(R_\bbX) \, \mid \,
f \cdot \fC^{\sigma}_\bbX \subseteq R_\bbX \, \},
$$
is called the \textbf{Dedekind different} of~$\bbX$
with respect to~$\sigma$.
\end{definition}

The following basic properties of the Dedekind different
of~$\bbX$ are shown in \cite[Proposition~3.7]{KL}.

\begin{proposition}\label{propSec4.7}
Let~$\sigma$ be a trace map of $Q^h(R_\bbX)/L_0$.

\begin{enumerate}
  \item[(a)]
  The Dedekind different $\delta^{\sigma}_\bbX$ is
  a homogeneous ideal of~$R_\bbX$ and $x_0^{2r_{\bbX}}
  \in \delta^{\sigma}_\bbX$.

  \item[(b)]
  The Hilbert function of~$\delta^{\sigma}_\bbX$ satisfies
   $\HF_{\delta^{\sigma}_\bbX}(i) = 0$ for $i < 0$,
   $\HF_{\delta^{\sigma}_\bbX}(i) = \deg(\bbX)$ for
   $i \ge 2r_{\bbX}$, and
   $0 \le \HF_{\delta^{\sigma}_\bbX}(0) \le \cdots
   \le \HF_{\delta^{\sigma}_\bbX}(2r_{\bbX}) = \deg(\bbX)$.
   In particular, the regularity index of~$\delta^{\sigma}_\bbX$
   satisfies $r_{\bbX} \le \ri(\delta^{\sigma}_\bbX)
   \le 2r_{\bbX}$.
\end{enumerate}
\end{proposition}

When $\bbX$ has CBP($d$), the Hilbert function of
the Dedekind different and its regularity index can be
described as follows. We use the notation
$\alpha_{\delta} = \min\{i\in\bbN
\mid (\delta^{\sigma}_\bbX)_i \ne \langle 0\rangle\}$.

\begin{proposition} \label{propSec4.8}
Let $K$ be an infinite field, let $\sigma$ be a trace map
of $Q^h(R_\bbX)/L_0$, and suppose that $\bbX$ has CBP($d$)
with $0 \le d \le r_\bbX-1$.
\begin{enumerate}
  \item[(a)]
  We have $d+1 \le \alpha_{\delta} \le 2r_{\bbX}$
  and $\HF_{\delta^{\sigma}_\bbX}(i) \le \HF_{\bbX}(i-d-1)$ for all $i\in\bbZ$.

  \item[(b)]
  Let $i_0$ be the smallest number such that
   $\HF_{\delta^{\sigma}_\bbX}(i_0) = \HF_{\bbX}(i_0-d-1) > 0$.
   Then we have $\HF_{\delta^{\sigma}_\bbX}(i)=\HF_{\bbX}(i-d-1)$
   for all $i \ge i_0$ and
   $$
   \ri(\delta^{\sigma}_\bbX) = \max\big\{\, i_0, r_\bbX+d+1 \,\big\}.
   $$
\end{enumerate}
\end{proposition}

\begin{proof}
Since $\fC_\bbX^\sigma \cong \omega_{R_\bbX}(1)$,
Corollary~\ref{corSec4.4} implies that there is
$g\in (\fC_\bbX^\sigma)_{-d-1}$ such that
$\Ann_{R_\bbX}(g)=\langle 0\rangle$. Notice that
$x_0$ is a non-zerodivisor of $R_\bbX$.
Then we find a non-zerodivisor $\widetilde{g}\in (R_\bbX)_{r_{\bbX}}$
such that $g=x_0^{-r_{\bbX}-d-1}\widetilde{g}$
by \cite[Proposition~3.7]{KL}.  Observe  that
$\widetilde{g}\cdot (\delta^{\sigma}_\bbX)_i
\subseteq x_0^{r_{\bbX}+d+1}(R_\bbX)_{i-d-1}$.
This implies $(\delta^{\sigma}_\bbX)_i = \langle 0\rangle$
for $i \le d$, and so $d+1 \le \alpha_{\delta}$.
Moreover, for all $i\in \bbZ$, we have
$$
\begin{aligned}
\HF_{\delta^{\sigma}_\bbX}(i) &= \dim_K(\delta^{\sigma}_\bbX)_i
= \dim_K (\widetilde{g}\cdot\delta^{\sigma}_\bbX)_i) \\
& \le \dim_K (x_0^{r_{\bbX}+d+1}(R_\bbX)_{i-d-1})
= \HF_{\bbX}(i-d-1).
\end{aligned}
$$
Thus claim~(a) is completely proved.

Now we prove claim~(b).
Clearly, we have $d+1 \le i_0 \le 2r_{\bbX}$.
By induction, we only need to show that
$\HF_{\delta^{\sigma}_\bbX}(i_0+1)=\HF_{\bbX}(i_0-d)>0$.
Let $f \in (R_\bbX)_{i_0-d}\setminus\{0\}$.
There are $g_0, \dots, g_n \in (R_\bbX)_{i_0-d-1}$
such that $f= x_0g_0 + x_1g_1 +\cdots+x_ng_n$.
By assumption, we have
$\widetilde{g}\cdot(\delta^{\sigma}_\bbX)_{i_0}
= x_0^{r_{\bbX}+d+1}(R_\bbX)_{i_0-d-1}$.
This enables us to write
$x_0^{r_{\bbX}+d+1}g_j = \widetilde{g}h_j$ for some
$h_j\in(\delta^{\sigma}_\bbX)_{i_0}$, where $j\in\{0,\dots,n\}$.
Thus we have
$$
\begin{aligned}
x_0^{r_{\bbX}+d+1}f
&= x_0^{r_{\bbX}+d+1} (x_0g_0 + x_1g_1 + \cdots + x_ng_n)
 = x_0\widetilde{g}h_0+x_1\widetilde{g}h_1+\cdots+x_n\widetilde{g}h_n \\
&= \widetilde{g} (x_0h_0 + x_1h_1 + \cdots + x_nh_n)
\end{aligned}
$$
and so $x_0^{r_{\bbX}+d+1}f \in
\widetilde{g}\cdot(\delta^{\sigma}_\bbX)_{i_0+1}$.
Hence $x_0^{r_{\bbX}+d+1}(R_\bbX)_{i_0-d} \!=\!
\widetilde{g}\cdot(\delta^{\sigma}_\bbX)_{i_0+1}$.
In other words, we get
$\HF_{\delta^{\sigma}_\bbX}(i_0+1)=\HF_{\bbX}(i_0-d)$.

Let $k = \max\big\{\, i_0, r_{\bbX}+d+1 \, \big\}$.
In order to prove the equality $\ri(\delta^{\sigma}_\bbX)=k$,
we consider the following two cases.

\medskip
\noindent\textbf{Case (1)}\quad Let $i_0 \ge r_{\bbX}+d+1$.
Then we have $k=i_0$. Observe that
$$
\deg(\bbX) \ge \HF_{\delta^{\sigma}_\bbX}(k) = \HF_{\bbX}(k-d-1)\ge
\HF_{\bbX}(r_{\bbX}) = \deg(\bbX).
$$
It follows that $\HF_{\delta^{\sigma}_\bbX}(k) = \deg(\bbX)$, and hence
$k \ge \ri(\delta^{\sigma}_\bbX)$.
Moreover, for $i < k=i_0$, we have $\HF_{\delta^{\sigma}_\bbX}(i)
< \HF_{\bbX}(i-d-1) \le \HF_{\bbX}(k-d-1) = \deg(\bbX)$.
Thus we get $\ri(\delta^{\sigma}_\bbX) = k$.

\medskip
\noindent\textbf{Case (2)}\quad Let $i_0 < r_{\bbX}+d+1$.
Then we have $k = r_{\bbX}+d+1$ and
$\HF_{\delta^{\sigma}_\bbX}(k) = \HF_{\bbX}(k-d-1)
= \HF_{\bbX}(r_{\bbX}) = \deg(\bbX)$.
This implies $k \ge \ri(\delta^{\sigma}_\bbX)$.
For $i < k$, we have $\HF_{\delta^{\sigma}_\bbX}(i) \le
\HF_{\bbX}(i-d-1) \le \HF_{\bbX}(r_{\bbX}-1) < \deg(\bbX)$.
Hence we obtain $\ri(\delta^{\sigma}_\bbX) = k$ again.
\end{proof}

In the special case that $\bbX$ is a locally Gorenstein
Cayley-Bacharach scheme, the regularity index
of the Dedekind different attains the maximal value.
This also follows from \cite[Proposition~4.8]{KLL}
with a different proof.

\begin{corollary}
In the setting of Proposition~\ref{propSec4.8},
assume that $\bbX$ is a Cayley-Bacharach scheme.
\begin{enumerate}
  \item[(a)] The regularity index of the Dedekind different
  $\delta^{\sigma}_\bbX$ is $2r_\bbX$.
  \item[(b)] The scheme $\bbX$ is arithmetically Gorenstein
  if and only if the Hilbert function
  of~$\delta^{\sigma}_\bbX$   satisfies
  $\HF_{\delta^{\sigma}_\bbX}(i)=\HF_{\bbX}(i-r_\bbX)$ for all
  $i\in\bbZ$.
\end{enumerate}
\end{corollary}
\begin{proof}
Claim (a) follows directly from the proposition,
and claim (b) follows by \cite[Proposition~5.8]{KL}.
\end{proof}

% -----------------------------------------------------------
\bigbreak
\subsection*{Acknowledgments.}
The second author thanks the University of Passau for its hospitality and support during part of the preparation of
this paper. The third and fourth authors
would also like to acknowledge the support from
the Vietnam National Foundation (TN-8).
%Last, but not least, we are extremely thankful to the referee
%for his very detailed and enlightening comments.

\bigbreak


\begin{thebibliography}{10}
\bibitem {ApC}
The ApCoCoA Team,
{\em ApCoCoA: Applied Computations in Commutative Algebra},
available at \texttt{http://apcocoa.uni-passau.de}.

\bibitem {BKM}
G.~Bolondi, J.~Kleppe, and R.~Mir\'{o}-Roig,
{\em Maximal rank curves and
singular points of the Hilbert scheme},
Compositio Math. {\bf 77} (1991), 269--291.

\bibitem {Cho}
K.~F.~E.~Chong,
{\em An application of liaison theory to the
Eisenbud-Green-Harris conjecture},
J. Algebra {\bf 445} (2016), 221-231.

\bibitem {DGO}
E.~D.~Davis, A.~V.~Geramita, and  F.~Orecchia,
{\em Gorenstein algebras and the Cayley-Bacharach theorem},
Proc. Amer. Math. Soc. {\bf 93} (1985), 593--597.

\bibitem {FGM}
G.~Favacchio, E.~Guardo, and J.~Migliore,
{\em On the arithmetically Cohen-Macaulay property for sets
of points in multiprojective spaces},
Proc. Amer. Math. Soc. {\bf 146} (2018), 2811--2825.

\bibitem {FPU}
L. Fouli, C. Polini, and B. Ulrich,
{\em Annihilators of graded components of the canonical
module, and the core of standard graded algebras},
Trans. Amer. Math. Soc. {\bf 362} (2010), 6183--6203.

\bibitem {GKR}
A.~V.~Geramita, M.~Kreuzer, and L.~Robbiano,
{\em Cayley-Bacharach schemes and their canonical modules},
Trans. Amer. Math. Soc. {\bf 339} (1993), 163--189.

\bibitem {GLS}
L.~Gold, J.~Little, and H.~Schenck,
{\em Cayley-Bacharach and evaluation codes on complete intersections},
J. Pure Appl. Algebra {\bf 196} (2005), 91--99.

\bibitem {GMN}
E.~Gorla, J.~C.~Migliore, and U. Nagel,
{\em Gr\"obner bases via linkage},
J. Algebra {\bf 384} (2013), 110--134.

\bibitem {Gua2000}
E.~Guardo,  {\em Schemi di ``Fat Points''},
Ph.D. Thesis, Universit\`{a} di Messina, 2000.

\bibitem {GW}
S.~Goto and K.~Watanabe, {\em On graded rings I},
J. Math. Soc. Japan {\bf 30} (1978),  179--213.

\bibitem {KMMNP}
J.~O.~Kleppe, J.~C.~Migliore, R.~Mir\'{o}-Roig,
U.~Nagel, and C.~Peterson,
{\em Gorenstein liaison, complete intersection liaison
invariants and unobstructedness},
Mem. Amer. Math. Soc. {\bf 154} (732) (2001), viii--116.

\bibitem {KL}
M.~Kreuzer and L.~N.~Long,
{\em Characterizations of zero-dimensional complete intersections},
Beitr\"age Algebra Geom. {\bf 58} (2017), 93--129.

\bibitem {KLL}
M.~Kreuzer, T.~N.~K.~Linh, and L.~N.~Long,
{\em The Dedekind different of a Cayley-Bacharach scheme},
J. Algebra Appl. (to appear 2018).

\bibitem {Kre}
M.~Kreuzer,
{\em On the canonical module of a 0-dimensional scheme},
Can. J. Math. {\bf 141} (1994), 357--379.

\bibitem {KLR}
M.~Kreuzer, L.~N.~Long and L.~Robbiano,
{\em On the Cayley-Bacharach property},
Commun. Algebra ( to appear 2018).

\bibitem {KR1995}
M.~Kreuzer and L.~Robbiano,
{\em On maximal Cayley-Bacharach schemes},
Commun. Algebra {\bf 23} (1995), 3357--3378.

\bibitem {KR2005}
M.~Kreuzer and L.~Robbiano,
{\em Computational Commutative Algebra 2},
Springer, Heidelberg 2005.

\bibitem {KR2016}
M.~Kreuzer and L.~Robbiano,
{\em Computational Linear and Commutative Algebra},
Springer Int. Publ. Switzerland, 2016.

\bibitem {Lam1998}
T.~Y.~Lam, {\em Lectures on Module and Rings},
Springer-Verlag, New York, 1991.

\bibitem {Lon}
L.~N.~Long, {\em Various differents for 0-dimensional schemes
and appplications}, Universt\"at Passau, Passau, 2015.

\bibitem {Mig}
J.~C.~Migliore,
{\em Introduction to Liaison Theory and Deficiency Modules},
Progr. Math. 165, Birkh\"{a}user, Boston, 1998.

\bibitem {MN}
J.~Migliore and U.~Nagel,
{\em Monomial ideals and the Gorenstein liaison class
of a complete intersection},
Compositio Math. {\bf 133} (2002), 25--36.

\bibitem {MR}
R.~Maggioni and A. Ragusa,
{\em The Hilbert function of generic plane sections
of curves of~$\bbP^3$},
Inv. Math. {\bf 91} (1988), 253--258.

\bibitem {PR}
S.~Popescu and K.~Ranestad,
{\em Surfaces of degree 10 in projective
fourspace via linear systems and linkage},
J. Algebraic Geom. {\bf 5} (1996), 13--76.

\end{thebibliography}
\end{document}